\documentclass[AApt, reqno]{amsart}
\usepackage{amsmath}
\usepackage{amsfonts}
\usepackage{graphicx}
\usepackage{url}
\usepackage{amssymb}
\usepackage{amsthm}
\usepackage{amscd}
\usepackage{epsfig}
\usepackage{comment}
\usepackage{hyperref}
\theoremstyle{plain}
\usepackage[all]{xypic}
\usepackage{tikz-cd}
\usepackage{fancyhdr}
\newtheorem{thm}{Theorem}[section]
\newtheorem{lemma}[thm]{Lemma}
\newtheorem{prop}[thm]{Proposition}
\newtheorem{cor}[thm]{Corollary}

\newtheorem*{thm*}{Theorem}
\newtheorem*{lemma*}{Lemma}
\newtheorem*{prop*}{Proposition}
\newtheorem*{cor*}{Corollary}
\newtheorem*{conj*}{Conjecture}

\theoremstyle{definition}
\newtheorem{defn}[thm]{Definition}

\newcommand{\ev}{\operatorname{ev}}

\theoremstyle{remark}

\numberwithin{equation}{section}

\def\bp{\mathbf{p}}
\def\bv{\mathbf{v}}
\def\bu{\mathbf{u}}

\def\bfw{\mathbf{w}}
\def\bC{\mathbf{C}}

\def\sp{\text{sp}}
\newcommand{\delt}[1]{\delta_{#1}}

\newcommand{\br}{\mathbf{r}}
\newcommand{\bq}{\mathbf{q}}
\newcommand{\rxy}{\mathbf{r}_{\scriptscriptstyle{X>Y}}}
\newcommand{\ryx}{\mathbf{r}_{\scriptscriptstyle{Y>X}}}
\newcommand{\axy}{\mathbf{a}_{\scriptscriptstyle{X>Y}}}
\newcommand{\ayx}{\mathbf{a}_{\scriptscriptstyle{Y>X}}}
\newcommand{\uxy}{\mathbf{u}_{\scriptscriptstyle{X>Y}}}
\newcommand{\vxy}{\mathbf{v}_{\scriptscriptstyle{X>Y}}}
\newcommand{\vx}{\mathbf{v}_{\scriptscriptstyle{X}}}
\newcommand{\vect}[1]{\mathbf{v}_{\scriptscriptstyle{#1}}}

\newcommand{\prof}{P_{\mathbf{C}}^{\lambda}}
\newcommand{\QC}{\mathbb{R}^{\mathbf{C}}}
\newcommand{\bw}{B^\lambda_{\mathbf{w}}}
\newcommand{\bwp}{B^\lambda_{\mathbf{w}}(\mathbf{p})}

\def\4tabloid(#1, #2, #3, #4){
	\begin{array}{|c|} \hline #1 \\ \hline #2  \\ \hline #3 \\  \hline #4 \\ \hline
	\end{array}
}

\newcommand{\hgt}{\operatorname{ht}}
\usepackage{color}

\newcommand{\clxy}[2]{\mathbf{C}_{\lambda}^{X=#1,Y=#2}}
\newcommand{\dsum}{\displaystyle \sum}

\begin{document}
	
	\title{The Euclidean geometry of \\ cardinal welfare functions}
	\author{Tim Ridenour and Prasad Senesi}

	\maketitle

\begin{abstract} 
	We exploit the standard inner product of Euclidean space to provide a new direction from which one can understand and analyze certain voting methods.  Using this perspective along with the action of the symmetric and special orthogonal groups on the vector space of profiles, we extend some natural voting criteria to ballots of arbitrary composition type.  
\end{abstract}
	
	\section{Introduction}

	The mathematical foundations of voting theory date back to the late 18th century, when Jean-Charles Borda introduced the Borda Count Method, and the Marquis de Condorcet wrote about his now well-known paradox. This discipline saw steady development leading to Kenneth Arrow's 1950 impossibility theorem \cite{arrow} for which he was later awarded a Nobel prize.  Subsequent decades saw further contributions to social choice theory, using a greater variety of mathematical tools of increasing sophistication, including - but not limited to - probability theory, geometry, topology, and algebra.  In the 1990's, Donald Saari, a mathematician and economist at UC Irvine, began to use geometric methods to investigate the structure of voting systems and to understand and generate paradoxes of voting theory. Among his many effective approaches, one was the representation of profiles (collections of ballots) as elements of a vector space. More recently, Michael Orrison and his colleagues at Harvey Mudd College have recast some of Saari's geometric methods in an algebraic framework (\cite{orrison}, \cite{Crisman}) and employed the action and representation theory of the symmetric group on vector spaces of profiles. 		
		
	Here, we rely heavily - and build - upon the framework introduced by Saari and Orrison in \cite{saari} and \cite{orrison}. In particular, a significant portion of the vector space formalism we describe here is either explicitly or implicitly found in \cite{orrison}.  Foremost, we are concerned not only with fully-ranked ballots of candidates, but with a vector space of `partially ranked' ballots of arbitrary composition type $\lambda = (\lambda_1, \ldots, \lambda_m)$, where $\lambda_i > 0$ and $\sum \lambda_i = |\mathbf{C}|$ (where $\mathbf{C}$ is the collection of all candidates).  These are ballots for which there are $\lambda_1$ first--place candidates, $\lambda_2$ second--place candidates, and so on.  Some practical applications for such partially ranked ballots can be found in \cite{urken}.

The well-known family of positional voting methods plays a central role here; indeed, we prove that all neutral linear cardinal welfare functions are of this form - and repeatedly utilize this fact.  At the same time, the most obvious limitation of our scope is that, for most of our results, we are restricted to these linear functions.  Previous important contributions to social choice theory, from K. Arrow (\cite{arrow},\cite{arrow2}), A. Sen (\cite{sen77},\cite{sen79}), and K. Suzumura (\cite{suzumura83},\cite{suzumura16}), to name only a few, make no such restrictions.  Nevertheless, there is a compelling presence of algebraic and geometric structure in this limited context. 

	The results in \cite{orrison} are built upon a novel use of the representation theory of the symmetric group on the space of voting profiles.  While this group action plays a role here, our contribution is the emphasis and exploitation of a set of distinguished `results vectors' $\vx$ which, by their inner products with profile vectors, completely determine the outcome of an election using a positional voting method. Our more elementary approach is to exploit the Euclidean geometry - and inner product - of the vector space of profiles and of election outcomes. The geometry of this interplay provides a novel and illuminating perspective from which we can examine the behavior of linear voting methods.

	As a proof-of-concept for our formalism, we provide a uniform proof that (almost) all neutral linear social choice functions fail to satisfy the ubiquitous \textit{independence of irrelevant alternatives} criterion.  This is a well-known result for \textit{all} neutral social choice functions (as a consequence of Arrow's Theorem), but only for fully--ranked ballots.  Our results provide a restricted (only to linear functions) extension to all partially ranked ballots.  Also in the generality of partially ranked ballots, we provide criteria for arbitrary neutral linear social choice functions to satisfy Pareto efficiency and a suitably defined \textit{strong majority criterion}.  
	
	 There is undoubtedly much more to say than we have here.  We are hopeful that more representation--theoretic tools can be introduced in this framework, perhaps to understand some more nuanced questions concerning insincere voting and the potential for manipulation of voting methods. The role of the vectors $\vxy$, and their geometry, have some interesting combinatorial content.  Further investigations of these vectors will be addressed in a forthcoming publication.
	 
	\section{Preliminary definitions and notation}
	
	We fix an integer $n$ (the number of candidates), and let $\mathbf{C} = \left\{ C_1, \ldots, C_n \right\}$ (the enumerated list of candidates).  Let $\mathbb{Z}_+$ ($\mathbb{R}_+$) be the set of nonnegative integers (real numbers).  For a positive integer $r$, we set $\left[ r \right] = \left\{ 1, \ldots, r \right\}$.  In the following, for any finite set $S$ we denote by $\mathbb
{R}^S$ the vector space of functions from $S$ to $\mathbb{R}$; by enumerating a basis, we obtain an isomorphism $\mathbb{R}^S \cong \mathbb{R}^{| S |}$.

		Let $\lambda = \left( \lambda_1, \ldots, \lambda_m \right)$ be a composition of $n$.  Henceforth, $|\lambda| $ will denote the number of elements $\lambda_1, \ldots, \lambda_m$ of $\lambda$, and $k$ will denote the multinomial coefficient $\displaystyle{\frac{n!}{\prod \lambda_i !}}$. Let $\mathbf{C}_\lambda$ be the set of all tabloids of shape $\lambda$ which are obtained by labeling the corresponding Young diagram with the candidates $C_1, \ldots, C_n$.  Alternatively $\mathbf{C}_\lambda$ consists of all functions $b: \mathbf{C} \rightarrow \left[ m \right] $ such that $| b^{-1}(i)| = \lambda_i$.  For a fixed candidate $X \in \mathbf{C}$, we define the evaluation map $\ev_X: \mathbf{C}_\lambda \rightarrow \left[ m \right]$ by $\ev_X(b) = b(X)$.  
		
		\subsubsection*{Example}  Let $\mathbf{C} = \left\{ A, B, C, D, E, F \right\}$, and $\lambda = (2, 1, 3)$.  Then $| \lambda| = 3$, and $\prof$ is spanned by $\mathbb{R}$--linear combinations of tabloids of the form 
		\[   b_1 =  {\begin{tabular}{| c  c  c |} 
 \cline{1-2}
 B & D & \multicolumn{1}{| c}{ } \\ 
 \cline{1-3}  
A & C & F \\ 
\cline{1-3}
\multicolumn{1}{| c |}{E}
\\ \cline{1-1}
\end{tabular}}\; , \; b_2 =  {\begin{tabular}{| c  c  c |} 
 \cline{1-2}
 A & B & \multicolumn{1}{| c}{ } \\ 
 \cline{1-3}  
E & F & C \\ 
\cline{1-3}
\multicolumn{1}{| c |}{D}
\\ \cline{1-1}
\end{tabular}}\; , \ldots \]  
We have $\ev_D(b_1) = 1$ and $\ev_D(b_2) = 3$, for example.
	\\
	\mbox{} \hfill \textit{End of example.}
	\\

		Let $ P_{ \mathbf{C}, \lambda} = \displaystyle{ \mathbb{R}^{\mathbf{C}_\lambda} \cong \bigoplus_{b \in \mathbf{C}_\lambda} \mathbb{R}b }$.  We refer to an element of $\mathbf{C}_\lambda$ as a \textit{ballot} and an element of $P_{\mathbf{C}, \lambda}$ as a \textit{profile}.   Any element $\bp \in P_{\mathbf{C}, \lambda}$ is a map $\bp: \mathbf{C}_\lambda \rightarrow \mathbb{R}$; the scalar $\bp(b)$ is the \textit{coefficient} of $b$ in $\bp$.  For example, if $\mathbf{C} = \left\{ A, B, C, D, E, F \right\}$ and $\lambda = (2, 3, 1)$, a typical profile looks like 
		
		\[ 0.7 \; {\begin{tabular}{| c  c  c |} \cline{1-2}
A & F & \multicolumn{1}{| c}{ } \\ \cline{1-3}  
B & C & D \\ \cline{1-3}
\multicolumn{1}{| c |} E  \\ \cline{1-1}
\end{tabular}}  + 5 \; {\begin{tabular}{| c  c  c |} \cline{1-2}
B & E & \multicolumn{1}{| c}{ } \\ \cline{1-3}  
A & F & D \\ \cline{1-3}
\multicolumn{1}{| c |} C  \\ \cline{1-1}
\end{tabular}} - 3 \;  {\begin{tabular}{| c  c  c |} \cline{1-2}
A & F & \multicolumn{1}{| c}{ } \\ \cline{1-3}  
B & C & D \\ \cline{1-3}
\multicolumn{1}{| c |} E  \\ \cline{1-1}
\end{tabular}}\; . \]

	  We say a profile $\bp$ is \textit{nonnegative} if $\bp(b) \geq 0$ for all ballots $b$.  	 If $\bp$ is nonnegative, we should think of $\bp$ as a pre-sorted `collection of ballots', wherein the presence of non-integer coordinates represents the notion that each voter can `split' her vote among the ballots.   Our definition of a profile does not allow us to distinguish ballots from individual voters; in fact these elements of $\prof$ should be (and are sometimes) called \textit{tabulated} profiles.  In this sense our profiles are anonymous, and their definition restricts our study here to anonymous voting methods (those which treat all candidates equally).  
	  
	  We identify $\mathbf{C}_\lambda$ with a basis of $\prof$ by associating to each $b \in \mathbf{C}_\lambda$ the delta function $\delta_{b}$, and we define an inner product on $\prof$ by 
		\[  \bp \cdot \bq = \sum_{ b\in \mathbf{C}_\lambda } \bp(b) \bq(b). \]

			 We note that any ballot naturally provides a pairwise ranking of candidates: the ballot $b$ `ranks $X$ over $Y$' if $b(X) < b(Y)$ (i.e., if $X$ appears above $Y$ in the corresponding tabloid).  If a ballot occurs with a negative coefficient ($\bp(b) < 0$), we will postulate that this ballot provides pairwise rankings which are opposite to $b$: if $b$ (which we can naturally identify with the profile $\delta_b$) ranks $X$ over $Y$, then $-\delta_b$ ranks $Y$ over $X$ - see \textbf{The fundamental geometric relations} in Section 4 below.  However, we do not identify $-\delta_b$ with any nonnegative profile.

			Let $X, Y$ be two candidates in $\mathbf{C}$, and let $ \mathbf{C}^{X>Y}_\lambda  = \left\{ b \in  \mathbf{C}_\lambda : b(X) < b(Y) \right\}$.  Note the alternate directions of the inequalities: the superscript of $\mathbf{C}^{X>Y}_\lambda $ reflects the fact that $X$ is preferred over (`better than') $Y$, but this is true for a ballot if the ballot maps $X$ to an integer \textit{less than} $Y$ (hence the condition $b(X) < b(Y)$). The sets $\mathbf{C}^{X=Y}_\lambda $, $\mathbf{C}^{X<Y}_\lambda $ are defined similarly.  We define the subspace $P_{\mathbf{C}, \lambda}^{X>Y}$ of $P_{\mathbf{C}, \lambda} $ as
		\[ P_{\mathbf{C}, \lambda}^{X > Y} = \text{span} \left\{ \delta_b :  b \in \mathbf{C}^{X>Y}_\lambda  \right\} . \]
		  We note here that a profile in $P_{\mathbf{C}, \lambda}^{X > Y}$ does not necessarily rank $X$ over $Y$: negative coefficients, which `reverse' the preference of $X$ to $Y$,  may be present.
		The subspaces $P_{\mathbf{C}, \lambda}^{Y>X}$, $P_{\mathbf{C}, \lambda}^{X=Y}$ are defined similarly.  Clearly we have a vector space direct sum $P_{\mathbf{C}, \lambda} = P_{\mathbf{C}, \lambda}^{X>Y} \oplus P_{\mathbf{C}, \lambda}^{X<Y} \oplus P_{\mathbf{C}, \lambda}^{X=Y} $, and we denote by $\pi_{X > Y}$, $\pi_{X < Y}$, and $\pi_{X=Y}$ the corresponding projections onto these subspaces.  
		
		We define some distinguished vectors which will play a significant role in arguments to follow.  We let $\mathbf{1} := \sum_{b} \delta_b$, the \textit{unit profile}.  For any two candidates $X, Y$, we define the profile vectors $\axy$ and $\rxy  $  as 
\[ \axy = \sum_{b(X) < b(Y)} \delta_b, \; \; \text{ and } \; \;  \rxy =\axy - \ayx . \] 
We collect a few basic properties of these vectors which follow directly from their definitions: 
\begin{prop}\label{rproperties} Let $X$ and $Y$ be any candidates.  Then 
\begin{enumerate}
\item $\rxy \perp \mathbf{1}$. 
\item $\rxy = - \ryx$. 
\item For any $\bq \in \mathbb{R}\mathbf{1}$, $\bq \cdot \axy = \bq \cdot \ayx$. 
\end{enumerate}
\end{prop}

\subsubsection*{Example}  Let $\lambda = (2,2)$.  We can enumerate the elements of $\mathbf{C}_\lambda$: 

\begin{center} 
\parbox{3 in}{		

\[  \begin{array}{cccccc}
b_1   &  b_2	& b_3 & b_4 & b_5 & b_6 \\
\begin{tabular}{|  c  c |} 
 \cline{1-2}
  A & B  \\ 
 \cline{1-2}  
C & D  \\ 
\cline{1-2}
\end{tabular}
 & 
\begin{tabular}{|  c  c |} 
 \cline{1-2}
   A&   C  \\ 
 \cline{1-2}  
B & D  \\ 
\cline{1-2}
\end{tabular}
& 
\begin{tabular}{|  c  c |} 
 \cline{1-2}
  A & D  \\ 
 \cline{1-2}  
B & C  \\ 
\cline{1-2}
\end{tabular}
&
\begin{tabular}{|  c  c |} 
 \cline{1-2}
  B & C  \\ 
 \cline{1-2}  
A & D  \\ 
\cline{1-2}
\end{tabular}
&
\begin{tabular}{|  c  c |} 
 \cline{1-2}
  B & D  \\ 
 \cline{1-2}  
A & C  \\ 
\cline{1-2}
\end{tabular}
&
\begin{tabular}{|  c  c |} 
 \cline{1-2}
  C & D  \\ 
 \cline{1-2}  
A & B  \\ 
\cline{1-2}
\end{tabular}
\end{array} 
\] 	

With this ordered basis, we have 
\[ \mathbf{r}_{\scriptscriptstyle{A>B}} = ( 0, 1, 1, -1, -1, 0), \; \; \; \mathbf{r}_{\scriptscriptstyle{A>C}} = ( 1, 0, 1, -1, 0, -1), \; \; \; \ldots  \] 

}
\end{center}

		For a profile $\mathbf{p} \in \prof$ and a ballot $b \in \mathbf{C}_\lambda$, $\left| \mathbf{p}(b) \right|$ is the `number of people who submitted the ballot $b$'.  The \textit{height} of a profile $\mathbf{p}$ is defined as  $\hgt(\mathbf{p}) = \sum_{b\in \mathbf{C}_\lambda} \left| \bp(b) \right|$; this integer counts the `number of voters submitting ballots'.   The \textit{space of $\bC$--orderings}  is the vector space $\QC$, as each $\lambda \in \QC$ provides a weak ordering $\preceq_{\lambda} $ on the set of candidates: for $\lambda \in \QC$, $X \preceq_{\lambda} Y$ if and only if $\lambda(X) \leq \lambda(Y)$.  
		
		A \textit{cardinal welfare function} (CWF) is a map $F: \prof \rightarrow \QC$.   By the above comment, any CWF provides, for a profile $\bp$, a weak ordering of the candidates.  Our terminology is motivated by the standard terminology in the literature, where a \textit{social choice function} is traditionally a function on $\mathbb{Z}_+$--linear combinations of $\mathbf{C}_{(1,\ldots, 1)}$ which returns a subset of $\mathbf{C}$ (the `winners' of the election), and a \textit{social welfare function} is a function on the same domain which returns a weak ordering on $\mathbf{C}$.  The CWFs we define provide more than just a weak ordinal ranking; they tell us \textit{how much} one candidate is preferred over another.  Any cardinal ranking provides an ordinal ranking; hence, any CWF provides a social welfare function.  And any social welfare function provides a social choice function (simply by choosing the `top--ranked' candidates).  We will refer to any of these functions as a \textit{voting method}.  
		
		\subsection*{Examples of various voting methods}
		
		\begin{itemize}
			\item Let $\lambda = (1,\ldots, 1)$.  The function with domain $\bigoplus_{b \in \mathbf{C}} \mathbb{Z}_+ \delta_b$ which selects the candidate(s) with a plurality of first--place votes (the so-called \textit{plurality method}) is a social choice function (but is not a social welfare function).  
			\item Let $\lambda= (1, 1, \ldots, 1)$, with $| \mathbf{C}| = n$.  Define a function $F$ that awards points to a candidate $X$ as follows: for a ballot $b$,  
					\[  F(\delta_b)(X) =  \begin{cases}
						n, & b(X) = 1 \\
						n-1, & b(X) = 2 \\
						 \vdots &  \; \; \; \; \vdots \\
						1, & b(X) = n
					\end{cases} \] 
				Any $\mathbb{Z}_+$--linear combination of ballots $\bp$ then awards a point total to each candidate, calculated by summing over each ballot in $\bp$.  The function with domain $\bigoplus_{b \in \mathbf{C}} \mathbb{Z}_+ \delta_b$ which assigns to each candidate their corresponding point total is a CWF.  This is the well--known \textit{Borda Count method}, used in some real--world elections, including parliamentary elections of Nauru, and the selections of the Heisman trophy winner and NFL MVP.  The Borda Count method is a CWF.
				\item Let $X \in \mathbf{C}$, and define $F(\bp) = X$ for all $\bp \in \prof$.  This constant function at $X$ is sometimes called \textit{monarchy}.  This is a social choice function.
				\item If a candidate has a majority of first--place votes, declare that candidate the winner.  Otherwise, remove from the ballots the candidate with the least number of first place votes (preserving all remaining ordinal relations), and look again for a majority candidate.  Repeat this process until a majority candidate is found (or the race ends in a tie between two candidates).  This is \textit{Hare's Method}.  This is also a social choice function.
		\end{itemize}
				We refer the reader to \cite{robinson} for more examples of voting methods.  
		
		Although we do not require CWFs to be linear, our focus here will be on those that are.  We will say $F$ is \textit{trivial} if, for all $\bp \in \prof$, $F(\bp) \in \mathbb{R}\mathbf{1}$ (the trivial CWFs are just those which return an all--way tie for any profile).  By the definitions given for $\mathbb{R}^{\mathbf{C}}$ and $\prof$, any CWF treats all voters equally; i.e., it is an \textit{anonymous} voting method.

		  For $\bv, \bfw \in \mathbb{R}^k$, we will denote the angle between $\bv$ and $\bfw$ by $\angle (\bv, \bfw)$.  The orthogonal complement of $\mathbb{R}\mathbf{v}$ is 
	\[ \mathbb{R}\mathbf{v}^\perp = \left\{ \mathbf{w} \in \mathbb{R}^k : \mathbf{v} \cdot \mathbf{w} = 0 \right\}, \]
which we will denote by $\mathbf{v}^\perp$ for brevity.  Let $(\mathbf{v})_+ = \left\{ \mathbf{w} \in \mathbb{R}^k : \mathbf{v}\cdot \mathbf{w} > 0 \right\}$; this is the \textit{positive half--space} associated with $\mathbf{v}$.  Let $(\bv)_-  = - (\bv)_+$.  For any $\mathbf{v} \in \mathbb{R}^k$, we have a disjoint union  
	$  \mathbb{R}^k = \mathbb{R}\mathbf{v}_+ \cup \mathbb{R}\mathbf{v}^\perp \cup \mathbb{R}\mathbf{v}_- $. These half--spaces will play a prominent role in our description and analysis of voting criteria.  We state several elementary but relevant properties here: 
	
	\begin{prop}\label{halfspaces}
	Let $V$ be a vector space, and $\bv, \bfw \in V$.  The following are equivalent: 
	\begin{enumerate}
		\item $\bv \parallel \bfw$ 
		\item $(\bv)_+ = \pm (\bfw)_+$
		\item $(\bv)_+ \subseteq  \pm (\bfw)_+$
	\end{enumerate}
	\end{prop}
	\begin{proof}
		Let $\bv \in V$ and $\bv = \bv_1 + \bv_2$ be a decomposition of $\bv$ into two orthogonal components (so that $\bv_1 \cdot \bv_2 = 0$).  If $\bv_1 \neq 0$, then $\bv \cdot \bv_1 = (\bv_1 + \bv_2) \cdot \bv_1 = | \bv_1 |^2 > 0$, and similarly for $\bv_2 $.   
		
		The proofs of (1) $\Rightarrow$ (2) and  (2) $\Rightarrow$ (3) are immediate.
To prove (3) $\Rightarrow$ (1), assume (3) and suppose $\bv \nparallel \bfw$.  Let $\bv = \bv_1 + \bv_2$ be the decomposition of $\bv$ onto $\mathbb{R} \bfw$ and $\bfw^\perp$, respectively, so that $ 0 \neq \bv_2 \in \bfw^\perp$.  Then, since $\bv_2 \cdot \bv > 0$, by (3) we must have $\bv_2 \cdot \bfw \neq 0$, a contradiction. 	
			\end{proof}
\noindent A very minor modification of the proof shows, of course, that 
\[ \bv \in \mathbb{R}_{> 0} \bfw  \; \; \Leftrightarrow \; \;  (\bv)_+ \subseteq (\bfw)_+  \; \; \Leftrightarrow \; \;  (\bv)_+ = (\bfw)_+. \]

 \section{Group actions}
  If $X$ is a set, we denote by $\text{Aut}(X)$ the collection of bijections $X \rightarrow X$ which preserve some structure of $X$.  In particular, if $X$ is a finite set with no additional structure, $\text{Aut}(X)$ is the group of all bijections of $X$.  If $X$ is a vector space, $\text{Aut}(X)$ is the group of all invertible linear maps on $X$.   An action of a group $G$ on $X$ is a group homomorphism $G \rightarrow \text{Aut}(X)$.  If $X$ is a vector space, this provides us with a \textit{representation} of $G$.  There are two group actions we will use.    In all cases, we will denote by $g.x$ the action of a group element $g$ acting on $x$.  
 For an element $x\in X$, we let $G^x = \left\{ g \in G : g.x = x \right\}$ (the \textit{isotropy subgroup} of $x$ in $G$). If $X$ is a vector space, a subspace $W$ of $X$ is \textit{invariant} under $G$ if $g.W \subseteq W$ for all $g \in G$. \subsubsection*{Permutation of candidates}    Let $S_{\mathbf{C}}$ be the symmetric group on the set $\mathbf{C}$ of candidates.  Then $S_{\mathbf{C}}$ acts naturally on $\mathbf{C}_\lambda$ by $\tau.b = b \circ \tau$, hence on $P_{\lambda, \mathbf{C}}$, and also on $\mathbb{R}^{\mathbf{C}}$.

\subsubsection*{Example}  Let $\lambda = (2,2)$, and enumerate the elements of $\mathbf{C}_\lambda$ as in the example above:

\begin{center} 
\parbox{5 in}{		

\[  \begin{array}{cccccc}
b_1   &  b_2	& b_3 & b_4 & b_5 & b_6 \\
\begin{tabular}{|  c  c |} 
 \cline{1-2}
  A & B  \\ 
 \cline{1-2}  
C & D  \\ 
\cline{1-2}
\end{tabular}
 & 
\begin{tabular}{|  c  c |} 
 \cline{1-2}
   A&   C  \\ 
 \cline{1-2}  
B & D  \\ 
\cline{1-2}
\end{tabular}
& 
\begin{tabular}{|  c  c |} 
 \cline{1-2}
  A & D  \\ 
 \cline{1-2}  
B & C  \\ 
\cline{1-2}
\end{tabular}
&
\begin{tabular}{|  c  c |} 
 \cline{1-2}
  B & C  \\ 
 \cline{1-2}  
A & D  \\ 
\cline{1-2}
\end{tabular}
&
\begin{tabular}{|  c  c |} 
 \cline{1-2}
  B & D  \\ 
 \cline{1-2}  
A & C  \\ 
\cline{1-2}
\end{tabular}
&
\begin{tabular}{|  c  c |} 
 \cline{1-2}
  C & D  \\ 
 \cline{1-2}  
A & B  \\ 
\cline{1-2}
\end{tabular}
\end{array} 
\] 	
\\
Let $\tau = (A \; B) \in S_{\mathbf{C}}$.  Then $\tau$ acts as follows: on $\mathbf{C}_\lambda$, 
\[ \tau(b_1) = b_1, \; \; \tau(b_2) = b_4, \; \; \tau(b_3) = b_5, \; \; \tau(b_6) = b_6.\]
On $\prof$, with respect to the above fixed basis, $\tau$ acts via matrix multiplication by
\[  \begin{bmatrix}
	 1 & 0 & 0 & 0 & 0 & 0 \\
	 0 & 0 & 0 & 1 & 0 & 0 \\
	 0 & 0 & 0 & 0 & 1 & 0 \\
	 0 & 1 & 0 & 0 & 0 & 0 \\
	 0 & 0 & 1 & 0 & 0 & 0 \\
	 0 & 0 & 0 & 0 & 0 & 1  
\end{bmatrix}.\]  
}
\end{center}

\subsubsection*{Rotations of the profile space}  The special orthogonal group $SO(k)$ is the group of all orientation--preserving isometries of $\mathbb{R}^k$.  It can be identified with collection of $k$-square matrices $A$ satisfying $A^TA = I = A A^T$ and $\det(A) = 1$.  The geometric properties of $SO(k)$ are well-known; see \cite{grove} or \cite{stillwell}, for example. As $k = \dim\left( \prof \right)$, $SO(k)$ acts (as `generalized rotations', see below) on $\prof$.  This action will be exploited in Section \ref{criteriaresults}.  We collect several important properties of the action of $SO(k)$ on $\mathbb{R}^k$ in the following proposition.  
	\begin{prop}\label{rotationproperties} \mbox{}
	\begin{enumerate}
	\item  Action by $SO(k)$ preserves the Euclidean norm, inner product, and angle between vectors in $\prof$.
 \item 	Suppose $G$ is a subgroup of $SO(k)$, and $V$ is a subspace of $\mathbb{R}^k$.  Then $V$ is $G$--invariant if and only if $V^\perp$ is $G$--invariant.  In particular, $\mathbb{R}\mathbf{v}$ is a $G$--invariant subspace of $\prof$ if and only if $\mathbb{R}\mathbf{v}^\perp$ is a $G$--invariant subspace.
 \item Let $T \in SO(k)$.  If $T\left( (\bv)_+ \right) \subseteq (\bv)_+$, then $T(\bv) = \bv$. 
 \item If $\br \in \mathbb{R}^k$, $\br \neq 0$, the only fixed points of the isotropy subgroup $SO(k)^\br$ are the scalar multiples of $\br$.  For $T \in SO(k)^\br$ and $\bv \in \mathbb{R}^k$, $\angle(\bv , \br) = \angle(T(\bv), \br)$.   
\end{enumerate}
	\end{prop}
	\begin{proof}
		(1) is well-known. We prove (2), (3) and (4).  For (2), suppose that $V$ is $G$--invariant, and let $\bv \in V$.   Let $T \in G$.  Then $T^{-1} \in G$, and for any $\bfw \in V^\perp$, we have 
		\[ 0 = \bv \cdot \bfw =\left( T^{-1}.\bv  \right) \cdot \bfw = \bv \cdot \left( T.\bfw \right).  \]
		
		\pagebreak 
		Hence, $V^\perp$ is $G$--invariant. The opposite direction now follows from the equality $V = \left( V^\perp \right)^\perp$.

		To prove (3), we first claim that $(T.\bv)_+ \subseteq T(\bv_+)$.  Indeed, if $\bu \in (T.\bv)_+$, we have
		\[ \left( T^{-1}. \bu \right) \cdot \bv = \bu \cdot \left( T.\bv \right) > 0; \] 
		hence, $\bu \in T(\bv_+)$.  Therefore $(T. \bv)_+ \subseteq (\bv)_+$, and by Proposition \ref{halfspaces}, we have $T.\bv = c \bv$ for some $c > 0$.  Since $T$ is an isometry, we must have $c = 1$.

	   To prove (4), we can rotate $\br$ so that $\br = (0, \ldots, 0, 1)$, and identify $SO(k)^\br$ with the subgroup of matrices 
	   \[  \begin{bmatrix}
	   	\tilde{A} & 0 \\ 
	   	0 & 1 
	   \end{bmatrix},\] 
	   where $\tilde{A}$ is an element of $SO(k-1)$.  Indeed, any such matrix is an element of $SO(k)$ and fixes $\br$.  Conversely if $A(\br) = \br$ then we can write $T$ in the form above for some $(k-1)$--square matrix $\tilde{A}$, and $A \in SO(k)$ guarantees that $\tilde{A}^T \tilde{A} = 1 = \tilde{A} \tilde{A}^T$, $\det(\tilde{A}) = 1$. Then the first statement in (4) follows from the fact that $SO(k-1)$ has no fixed points in $\mathbb{R}^{k-1}$ (see \cite{garrett} for details) .  For the second statement, it is sufficient to show that $\bv \cdot \br = T(\bv) \cdot \br$, which follows because $\br$ is a fixed point of $T$:
	   \[ T(\bv) \cdot \br = T(\bv) \cdot T(\br) = \bv \cdot \br.\]			\end{proof}
	\noindent Because of the properties listed in (4), we call $SO(k)^\br$ the \textit{rotation subgroup of $SO(k)$ with axis} $\mathbb{R} \br$.   The action of an element $g \in SO(k)^\br$ on $\mathbb{R}^k$ can be thought of as a rotation about (the axis determined by) $\br$.

		\section{Positional voting}  
		We define a family of CWFs, the \textit{positional voting methods}, as follows.  Let $W = \mathbb{R}^{\left[ m \right]} \cong \mathbb{R}^m$, and $\mathbf{w} \in W$.  For any candidate $X \in \mathbf{C}$, we define the $X$--positional vector  $\bv_{\scriptscriptstyle{\bfw, X}} := \mathbf{w} \circ \ev_{\scriptscriptstyle{X}}$, which is an element of $\prof$.  Henceforth the weight vector $\bfw$ will usually be fixed and so we will write $\vx = \bv_{\scriptscriptstyle{\bfw, X}}$ for notational convenience unless otherwise necessary.
		
		The $\mathbf{w}$--positional CWF $\bw$ is defined as follows: for $\bp \in \prof$, 
		\[ \bwp : X \mapsto \bp \cdot \vx . \]
		We note that $\bw$ is linear in both $\bp$ and in $\bfw$, and that for any $\bfw \in \mathbb{R}\mathbf{1}$, $\bw$ is a trivial CWF.

		\subsubsection*{Example}
		Let $\mathbf{C} = \left\{ A, B, C \right\}$,  $\lambda = (1,1,1)$.  Then 
		\[ \mathbf{C}_\lambda = \left\{ \; \begin{tabular}{| c  |} 
		\cline{1-1}
		A \\ 
		\cline{1-1}  
		B\\ 
		\cline{1-1}
		C\\
		\cline{1-1}
		\end{tabular} \;  ,   \; \begin{tabular}{| c  |} 
		\cline{1-1}
		A \\ 
		\cline{1-1}  
		C\\ 
		\cline{1-1}
		B\\
		\cline{1-1}
		\end{tabular} \; , \; \begin{tabular}{| c  |} 
		\cline{1-1}
		C \\ 
		\cline{1-1}  
		A\\ 
		\cline{1-1}
		B\\
		\cline{1-1}
		\end{tabular} \; , \; \begin{tabular}{| c  |} 
		\cline{1-1}
		B \\ 
		\cline{1-1}  
		A\\ 
		\cline{1-1}
		C\\
		\cline{1-1}
		\end{tabular} \; , \begin{tabular}{| c  |} 
		\cline{1-1}
		B \\ 
		\cline{1-1}  
		C\\ 
		\cline{1-1}
		A\\
		\cline{1-1}
		\end{tabular} \; , \begin{tabular}{| c  |} 
		\cline{1-1}
		C \\ 
		\cline{1-1}  
		B\\ 
		\cline{1-1}
		A\\
		\cline{1-1} 
		\end{tabular} \;  \right\}  .\] 
		
		\pagebreak 
		
		\noindent The first ballot $\begin{tabular}{| c  |} 
		\cline{1-1}
		A \\ 
		\cline{1-1}  
		B\\ 
		\cline{1-1}
		C\\
		\cline{1-1}
		\end{tabular}$ , for example, is the function which maps $A$ to 1, $B$ to 2, and $C$ to 3.  We fix an ordering $\mathbf{b}_1, \ldots, \mathbf{b}_6$ of $\mathbf{C}_\lambda$ given by the ordering in the set above.  This gives us an ordered basis of $P_{\mathbf{C}, \lambda}$; denote by $\rho: \prof \rightarrow \mathbb{R}^6$ the corresponding isomorphism. Let  $\mathbf{w} = (3,2,1)$.
	Then $\mathbf{v}_{\mathbf{w}, A} = (3, 3, 2, 2, 1, 1)$. Similarly we find $\mathbf{v}_{\mathbf{w}, B} = (2, 1, 1, 3, 3, 2)$, and $\mathbf{v}_{\mathbf{w}, C} = (1, 2, 3, 1, 2, 3)$.  In this case the matrix representation $T_\bfw$ of $\bw$ is   
	\[  T_{\mathbf{w}} = \begin{bmatrix} 
		 | & |  & | \\ 
		\mathbf{v}_{\mathbf{w}, A} & \mathbf{v}_{\mathbf{w}, B} & \mathbf{v}_{\mathbf{w}, C}\\
		| & |  & | 
		\end{bmatrix}= \begin{bmatrix}
	3&2&1 \\ 3& 1 & 2 \\ 2&1&3 \\ 2&3&1 \\ 1 & 3 
	& 2 \\ 1 & 2 & 3  
	\end{bmatrix}.
	\] 
	Let $\mathbf{p}$ be a profile given by 
	\[ \mathbf{p} = \sum_{i=1}^6 r_i \mathbf{b}_i.   \] 
	Then $ \mathbf{p}_\mathbf{w}(A) = \rho(\mathbf{p}) \cdot (3, 3, 2, 2, 1, 1) = 3r_1 + 3r_2 + 2r_3 + 2r_4 + r_5 + r_6 $. \\
	\\
	\mbox{} \hfill \textit{End of example.}

		Let $\vxy = \vect{X} - \vect{Y}$.  With respect to a profile $\bp$ and the positional method $\bw$, a candidate $X$ defeats a candidate $Y$ if and only if $\bw(\bp)(X) > \bw(\bp)(Y)$.  But this inequality is equivalent to $\bp\cdot \vect{X} > \bp \cdot \vect{Y}$, or $\bp \in (\vxy)_+$.    A significant portion of our geometric perspective is based upon the following relations, one of which we submit as a postulate and another that follows easily from the definitions.
		\\
		\\
		\noindent \textbf{The fundamental geometric relations.} Fix a positional voting method $\bw$ and let $\bp$ be a profile. \mbox{}
		\\
		\\
			\textbf{Postulate 1}. \textit{In the profile $\bp$, a candidate $X$ defeats a candidate $Y$ in a head--to--head race if and only if $\bp \in (\rxy)_+$. }\mbox{}
			\\
			\\
		\noindent	\textbf{Proposition 1}. \textit{In the profile $\bp$, the CWF $\bw$ awards  a candidate $X$ more points than a candidate $Y$ if and only if $\bp \in (\vxy)_+$.} \mbox{}\\
	
		Note that the vectors $\rxy$ are determined by $\lambda$, and are a built-in `feature' of the profile space $\prof$.  The vectors $\vxy$, however, are determined by the weight vector $\bfw$.  The moral of the fundamental geometric relations is this: between candidates $X$ and $Y$, to determine who defeats whom in a head-to-head race, or to determine who `wins' with a profile $\bp$ and a positional voting method $\bw$, we don't actually need to count votes or calculate the positional point totals for $X$ and $Y$; we just need to find the angles $\angle(\bp, \rxy)$ or $\angle(\bp, \vxy)$, respectively.  Candidate $X$ is victorious over $Y$ (in either sense) if and only if this angle is acute; i.e., if and only if $\bp$ lies in the positive half-space determined by $\rxy$ or $\vxy$.  
		
		In the diagram below, a vector $\vxy$ is shown, along with its hyperplane $\vxy^\perp$.  Also shown are three profiles $\bp_1, \bp_2, \bp_3$.  In the outcome $\bw(\bp_1)$, $X$ defeats $Y$; in $\bw(\bp_3)$, $Y$ defeats $X$; and in $\bw(\bp_2)$, $X$ and $Y$ tie.
 
		\begin{center} \includegraphics[scale=0.5]{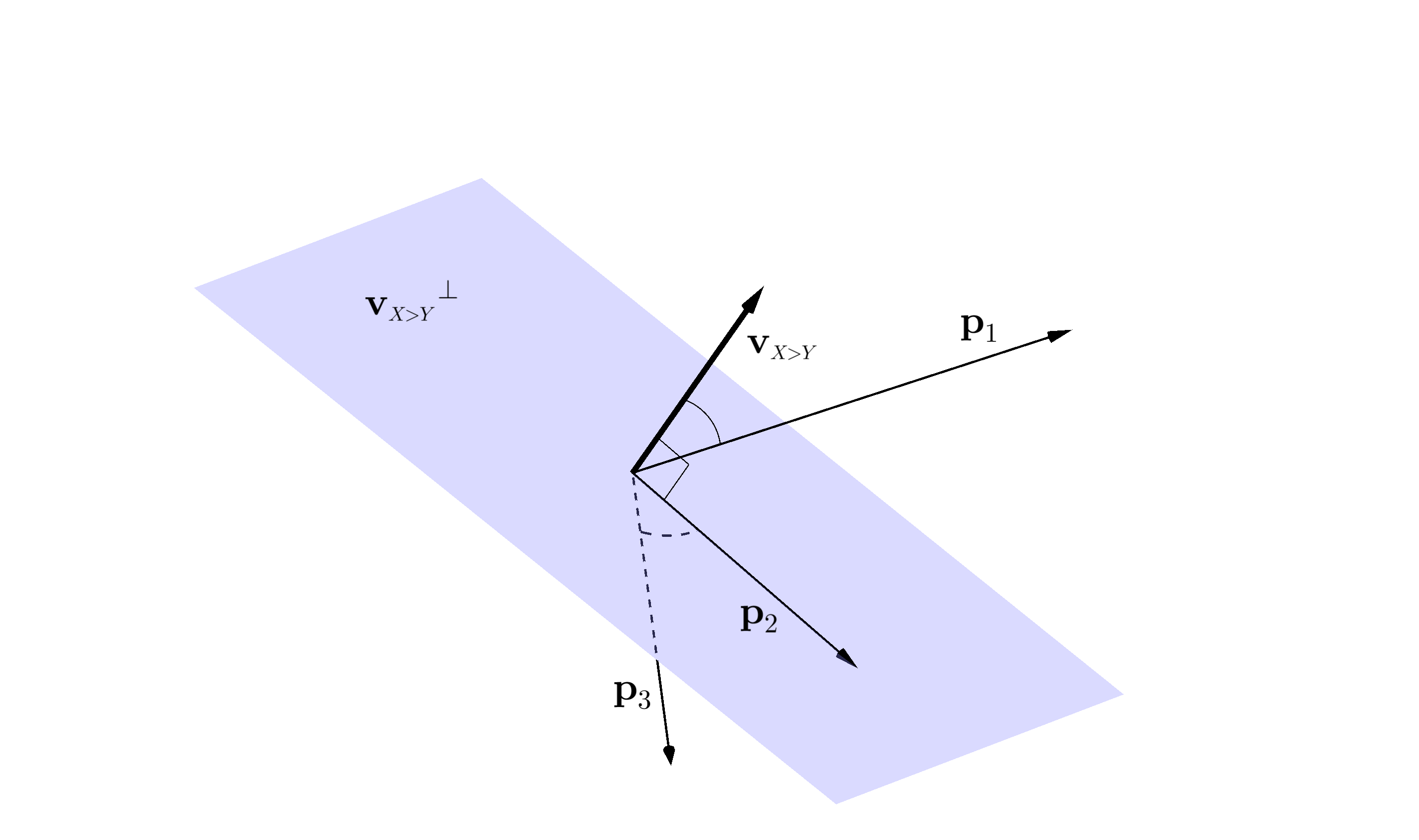} \end{center}
		\mbox{}\\
		We collect a few geometric facts concerning the vectors $\vxy$ and $\rxy$ :
		\begin{prop}\label{vproperties}  Let $X$ and $Y$ be distinct candidates.
		\begin{enumerate}  

			\item   $\mathbf{1} \perp \vxy$.
			\item   Suppose $|\lambda | = 2$.  Then $\vxy \parallel \rxy$. 
			\item   Suppose $| \lambda | > 2$.  If $\vxy \neq 0$, then $\vxy \nparallel \rxy$.  
		\end{enumerate}	
		\end{prop}
		\begin{proof} 
		To prove (1), note that since any transposition in $S_{\mathbf{C}}$ permutes $\mathbf{C}_\lambda$, we have 
		\[ \sum_{b \in \mathbf{C}_\lambda} w(b(X)) = \sum_{(X Y) b \, \in \, \mathbf{C}_\lambda} w(b(X)) = \sum_{b \in \mathbf{C}_\lambda} w(b(Y));\]
		hence, $\vxy \cdot \mathbf{1} =  \sum_{b \in \mathbf{C}_\lambda} w(b(X)) -\sum_{b \in \mathbf{C}_\lambda} w(b(Y)) = 0$.
		To prove (2) and (3), first suppose $| \lambda | = 2$.  In this case, for two candidates $X$ and $Y$ and a ballot $b$ we have either $b(X) =1$ and $b(Y) = 2$, $b(X) = 2$ and $b(Y)=1$, or $b(X) = b(Y)$.  Therefore 
	  \[ 
	  	\vxy(b) = \begin{cases}
	  		w(1) - w(2), & b(X) < b(Y) \\ 
	  		0,& b(X) = b(Y) \\
	  		w(2) - w(1), & b(X) > b(Y).
	  	\end{cases}\] 
	  Therefore $\vxy = c \, \rxy$, where $c= w(1) - w(2) $.  
	  
	  Next suppose $|\lambda| > 2$, and assume $\vxy \parallel \rxy$.  This is equivalent to 
	\[  \vxy(b) = \begin{cases}  c, & b(X) < b(Y) \\
 								 0,& b(X) = b(Y)\\
 								 -c, &  b(X) > b(Y)	
 \end{cases}\]
for some $c$.  This, in turn, gives us 
	\[  w(i) - w(j) = c , \; \; \text{for all} \; \; 1 \leq i < j \leq m,\]
	for which $c = 0$ is the only solution. 
		\end{proof}
 These profiles $\left\{ \vxy \right\}_{X, Y}$, and their orthogonality to $\mathbf{1}$, will be used in arguments to follow.

 A CWF $F: P_{\lambda, \mathbf{C}}\rightarrow \mathbb{R}^\mathbf{C}$ is \textit{neutral} if it is $S_{\mathbf{C}}$--equivariant; that is, if $F \circ \tau = \tau \circ F$ for all $\tau \in S_{\mathbf{C}}$.  Clearly any positional voting method is a neutral linear CWF.  Conversely, we gain no generality by considering arbitrary neutral linear CWFs: 
		\begin{thm}\label{linearscf}
			Any neutral linear CWF	  is a positional voting method.  
			\end{thm}
			\begin{proof}  		

					 Let $F$ be a neutral CWF.  For $ 1 \leq j \leq m$, let $N_{X,j} = \# \left\{  b \in \mathbf{C}_\lambda : b(X) = j \right\}$ (the number of ballots which place $X$ in $j^{\text{th}}$ place); then $N_{X, j} = N_{Y,j}$ for all $X, Y \in \mathbf{C}$. Define $\mathbf{w}_X \in \mathbb{R}^{\left[ m \right]}$ by 
					\[  \mathbf{w}_X(j) = \frac{1}{N_{X, j}} \left( \sum_{b \in \mathbf{C}_{\lambda}, b(X) = j} F(\delt{b})(X) \right). \]
					 	  For any $\tau \in 	S^X_{\mathbf{C}}$, we have 
			\[ F(\delt{b})(X) = F((\tau^{-1} \tau). \delt{b})(X) = F(\tau .\delt{b})(\tau. X) = F(\tau. \delt{b})(X). \]
			If $b'$ is any other ballot such that $b'(X) = b(X)$, we can find some $\tau \in S_\mathbf{C}^X$ such that $\tau b = b'$, and by the above comment we have $F(\delt{b})(X) = F(\delt{b'})(X)$.  Therefore $ \mathbf{w}_X(j) = F(\delt{b})(X)$, where $b$ is any ballot satisfying $b(X) = j$.
					
					Now let $Y$ be any other candidate, let $\sigma = (X \; Y) \in S_{\mathbf{C}}$, and $1 \leq j \leq m$.  Then
\begin{eqnarray*}
\mathbf{w}_Y (j) &=&  F(\delt{b})(Y)  \\
				&=&   F((\sigma^{-1} \sigma).\delt{b})(Y) \\
				&=&   F( \sigma.\delt{b})(X) \\			
				&=& F(\delt{b'})(X) \\
				&=& \mathbf{w}_X(j).
\end{eqnarray*}				
				
				So we can define $\mathbf{w} = \mathbf{w}_X$ for any candidate $
				X$, and then $ \bfw(b(X)) = F(\delta_b)(X)$ for any $b \in \mathbf{C}_\lambda$ and $X \in \mathbf{C}$.  Linearity of $F$ then gives us $F(\bp)(X) = \displaystyle{\sum \bp(b) (\bfw \circ \ev_X)(b)}$; hence, $F$ is the positional voting method with weight $\bfw$.  
				\end{proof}
				
\section{Results concerning voting criteria}\label{criteriaresults}
There are certain criteria or `fairness standards' one would hope any reasonable voting method should satisfy.  For an introductory discussion of these criteria, see \cite{robinson}.  For example, if candidate $X$ is ranked in first place for a majority of ballots, one might expect $X$ to be the unique winner of the election.  This condition is called the \textit{majority criterion}.  As compelling as it may seem, there are some prominent voting methods, such as the Borda count, that do not satisfy this condition.  A seminal result of economist Kenneth Arrow, proven in 1950, demonstrates that \textit{no} social welfare function can satisfy three particular prominent and compelling criteria -  anonymity, Pareto efficiency, and independence of irrelevant alternatives.  The first of these criteria, anonymity, is simply the condition that a voting method should not favor one member of the electorate over another.   All positional voting methods are, by design, anonymous, since the input data consists of a profile (which can be thought of as a ballot collection that has already been tabulated).  The remaining two criteria (Pareto efficiency and independence of irrelevant alternatives), along with several more, will be discussed below.  An introductory discussion of these and other voting criteria can be found in \cite{robinson}.  A comprehensive discussion of Arrow's Theorem, first published in \cite{arrow}, can be found in \cite{saari2}.  It is beyond our scope to discuss this Theorem any further.  Our more modest goal is to demonstrate the efficacy of the present framework to prove some useful facts concerning cardinal welfare functions and some voting criteria.  

\subsection{Independence of irrelevant alternatives}

Let $X$ and $Y$ be two candidates in an election with three or more candidates, and let $F$ be a CWF applied to a profile $\bp$.  Now suppose the voters are provided the opportunity of a `revote' to submit another ballot $\bq$ (perhaps some incriminating information was revealed about a candidate), except that no voter changed her mind with respect to the ordinal ranking of $X$ and $Y$; that is, if a voter ranked $X$ above $Y$ for $\bp$, she did the same for $\bq$, and vice versa.  For example, two voters may have voted, and then revoted, as follows: 
\begin{center}
  \[ \xymatrixcolsep{3pc}\xymatrix{
 {\begin{tabular}{| c  c  c |} 
\cline{1-2}
 \textbf{Y} & B & \multicolumn{1}{| c}{ } \\ 
 \cline{1-3}  
D & \textbf{X} & C \\ \cline{1-3}
  \multicolumn{1}{| c |}{ A} \\ \cline{1-1}
\end{tabular}}  \ar^{revote}[r] 
&   
{\begin{tabular}{| c  c  c |} \cline{1-2}
\textbf{Y} & A & \multicolumn{1}{| c}{ } \\ \cline{1-3}  
B & C & D \\ \cline{1-3}
\multicolumn{1}{| c |}{\textbf{X}}  \\ \cline{1-1}
\end{tabular}}
 } \]
 \[ \xymatrixcolsep{3pc}\xymatrix{
 {\begin{tabular}{| c  c  c |} 
 \cline{1-2}
 B & D & \multicolumn{1}{| c}{ } \\ 
 \cline{1-3}  
C & \textbf{X} & \textbf{Y} \\ 
\cline{1-3}
\multicolumn{1}{| c |}{A}
\\ \cline{1-1}
\end{tabular}}  \ar^{revote}[r] 
&   
{\begin{tabular}{| c  c  c |} \cline{1-2}
\textbf{X} & \textbf{Y} & \multicolumn{1}{| c}{ } \\ \cline{1-3}  
A & B & D \\ \cline{1-3}
\multicolumn{1}{| c |}{C} \\ \cline{1-1}
\end{tabular}}
 } \]
\end{center}

If we assume the individual ordinal preferences of $X$ and $Y$ remain unchanged for all ballots in the revote $\bp \rightarrow \bq$, we might expect a reasonable CWF to follow suit and leave the ordinal ranking of $X$ and $Y$ unchanged when evaluated at $\bp$ and $\bq$ ($F(\bp)(X) < F(\bp)(Y)$ if and only if $F(\bq)(X) < F(\bq)(Y)$, and $F(\bp)(X)  = F(\bp)(Y)$ if and only if $F(\bq)(X) = F(\bq)(Y)$).  In the example given directly above, the candidates A, B, C, and D are irrelevant to the ordinal ranking of $X$ and $Y$.  This expectation (stated precisely below) is therefore usually known as the criterion of \textit{independence of irrelevant alternatives} (henceforth abbreviated as IIA).  The criterion requires that the aggregate ordinal ranking of two candidates $X, Y$ remains invariant under any change in profile for which the individual ordinal rankings of $X$ and $Y$ remain the same, although the definitions found in the literature vary depending on context and application.  For example, the authors of \cite{robinson}, when discussing social choice functions on fully ranked ($\lambda = (1, \cdots, 1)$) profiles, define IIA as follows (although some terminology and notation has been adapted to match our own here):

\begin{center}
\parbox{4 in} { \textit{Suppose $X$ and $Y$ are two candidates, and $\bp$, $\bq$ are two profiles such that no voter changes their preference with respect to $X$ and $Y$ (so if $X$ is preferred to $Y$ for some ballot in $\bp$, then $X$ is preferred to $Y$ for the `revised' ballot in $\bq$, and vice versa).  If $X$ but not $Y$ is declared a winner when $F$ is evaluated at $\bp$, then $Y$ should not be declared a winner when $F$ is evaluated at $\bq$. } }	
\end{center}
\mbox{}\\
The definition found in 	\cite{saari2}, where ballots are still fully ranked, but social welfare functions are considered, is more general.   Here the author states that a social welfare function $F$ satisfies IIA if the following conditions holds:
\mbox{}\\
\begin{center}
\parbox{4 in} { 
\textit{Let $X$ and $Y$ be two candidates.  Suppose $\bp_1$ and $\bp_2$ are any two profiles where each voter's $\left\{ X \, Y \right\}$ ranking in $\bp_1$ agrees with the voter's $\left\{ X \, Y \right\}$ ranking in $\bp_2$.  Then the group's $\left\{ X \, Y \right\}$ ranking for $F(\bp_1)$ and $F(\bp_2)$ agree.}
}	
\end{center}
\mbox{}\\

We will recast this criterion in the current framework.  Our definition of the criterion will \textit{not} require each voter to preserve their ordinal ranking of $X$ and $Y$; indeed, as our profiles are already tabulated, we are prevented from even formulating such a condition.  Instead, we will impose the weaker implied condition that the \textit{number} of ballots ranking $X$ above $Y$ remains unchanged (and similarly for ballots ranking $Y$ above $X$).  There is no cost for this modification: all of our CWF's are, by design, anonymous, and so it can be shown that any CWF $F$ will satisfy the `traditional' IIA criterion if and only if it satisfies this weakened IIA criterion.   

Of course, traditionally profiles are collections of ballots, and there is no notion of a `negative vote'.  So at first we define this equivalence only on nonnegative profiles.   
 We will say that two nonnegative profiles $\bp$, $\bq$ are $X,Y$--equivalent, written $\bp \sim_{X,Y} \bq$, if the number of ballots ranking $X$ above $Y$ in $\bp$ is equal to the number of ballots ranking $X$ above $Y$ in $\bq$, and similarly for ballots ranking $Y$ above $X$.   We note that this notion of equivalence allows any voters with no preference between $X$ and $Y$ to enter or leave the electorate.  For example, we have 
\\
\\
$ \displaystyle{ 1\; {\begin{tabular}{| c  c  c |} 
\cline{1-1}
\multicolumn{1}{| c |}{\bf{X}}  \\ 
\cline{1-2}
A & B & \multicolumn{1}{| c}{ } \\ \cline{1-3}  
\bf{Y} & C & D \\ \cline{1-3}
\end{tabular}} 
 + 4 \;
{\begin{tabular}{| c  c  c |} 
\cline{1-1}
\multicolumn{1}{| c |}{A}  \\ 
\cline{1-2}
C & \bf{X} & \multicolumn{1}{| c}{ } \\ \cline{1-3}  
\bf{Y} & C & D \\ \cline{1-3}
\end{tabular}} 
 +  3 \;{\begin{tabular}{| c  c  c |} 
\cline{1-1}
\multicolumn{1}{| c |}{A}  \\ 
\cline{1-2}
\bf{X} & \bf{Y} & \multicolumn{1}{| c}{ } \\ \cline{1-3}  
B & C & D \\ \cline{1-3}
\end{tabular}} +  7 \;{\begin{tabular}{| c  c  c |} 
\cline{1-1}
\multicolumn{1}{| c |}{A}  \\ 
\cline{1-2}
\bf{Y} & B & \multicolumn{1}{| c}{ } \\ \cline{1-3}  
\bf{X} & C & D \\ \cline{1-3}
\end{tabular}} }$
\\
\\
$ \displaystyle{ \hspace{1 in} \; \; \; \; \; \; \sim_{X,Y} \; \; \; 
5 \;  {\begin{tabular}{| c  c  c |} 
\cline{1-1}
\multicolumn{1}{| c |}{A}  \\ 
\cline{1-2}
\bf{X} & B & \multicolumn{1}{| c}{ } \\ \cline{1-3}  
\bf{Y} & C & D \\ \cline{1-3}
\end{tabular}}  
 + 15 \;{\begin{tabular}{| c  c  c |} 
\cline{1-1}
\multicolumn{1}{| c |}{A}  \\ 
\cline{1-2}
B & C & \multicolumn{1}{| c}{ } \\ \cline{1-3}  
\bf{X} & \bf{Y} & D \\ \cline{1-3}
\end{tabular}}
 +  7\;{\begin{tabular}{| c  c  c |} 
\cline{1-1}
\multicolumn{1}{| c |}{\bf{Y}}  \\ 
\cline{1-2}
A & C & \multicolumn{1}{| c}{ } \\ \cline{1-3}  
B & D & \bf{X} \\ \cline{1-3}
\end{tabular}}.}$
\\
\\
Now we recast this condition in terms of profile vectors. 
\begin{prop}
	Two nonnegative profiles $\bp$, $\bq$ are $X,Y$--equivalent if and only if $\bp - \bq$ is orthogonal to the plane spanned by $\axy$ and $\ayx$.   
\end{prop}
\begin{proof}
Since $\bp$, $\bq$ are nonnegative, the condition that $\bp \sim_{X,Y} \bq$ is that the number of voters submitting ballots ranking $X$ over $Y$ remains unchanged, and similarly for $Y$ over $X$; i.e., we have 
 \[ \axy\cdot \bp = \text{ht} \left( \pi_{X>Y} (\bp) \right) = \text{ht} \left( \pi_{X>Y} (\bq) \right) = \axy \cdot \bq,  \] 
 and similarly for $\ayx$.  
 \end{proof}
 \noindent We use this proposition to extend $X,Y$--equivalence to all profiles, by defining $\bp   \sim_{X,Y} \bq$ if $\bp - \bq$ is orthogonal to $\sp_{\mathbb{R}} \left\{ \axy, \ayx \right\}$.  For a profile $\bp$, we will denote by $\left[ \bp \right]_{X,Y}$ the equivalence class of all profiles $\bq$ with $\bq \sim_{X,Y} \bp$. 

	For a positional voting method $\bw$, we now present our definition of the IIA criterion.  To motivate the definition, we clarify what we would desire: suppose, after a profile $\bp$ is submitted, a `revote' $\bq$ is provided.  If the number of ballots ranking $X$ above $Y$ remains unchanged in the revote, and similarly for $Y$ above $X$, we should hope that the aggregate ordinal ranking (under $\bw$) for $X$ and $Y$ remains unchanged.  At the same time, the number of ballots ranking $X$ and $Y$ equally should have no effect on this outcome.   The requirement that $\bw$ preserves the ordinal ranking on $X$ and $Y$ is equivalent to the condition 
	\[ \vxy \cdot \bp > 0 \; \; \Rightarrow  \; \; \vxy \cdot \bq > 0. \]
	Therefore we present the following definition for the IIA criterion, adapted to arbitrary compositions $\lambda$:
	\begin{defn}
		The positional voting method $\bw$ satisfies the IIA criterion if, for any profiles $\bp, \bq \in \prof$ with $ \bp \sim_{X,Y} \bq$, 
		\[ \bp \cdot \vxy > 0 \; \; \Rightarrow  \; \; \bq \cdot \vxy  > 0. \]
	\end{defn}
	
The next lemma concerns rotations in $\prof$ about $\rxy$ acting on the equivalence classes $\left[ \bp \right]_{X,Y}$.  Modulo the subspace $\mathbb{R}\mathbf{1}$, these equivalence classes are preserved by rotations about $\rxy$:
	\begin{lemma}\label{rotationlemma}
		If $T$ is a rotation about $\rxy$ and $\bp \in \prof$, then there exists $\bq \in \mathbb{R}\mathbf{1}$  such that 	$T(\bp) + \bq \in \left[ \bp \right]_{X,Y}$.\end{lemma}
	\begin{proof} 
		Let $\bp \in \prof$ and $T$ be any rotation of $\bp$ about $\rxy$, and define 
		\[ \bq = \left( \frac{\left(  \bp -T(\bp)\right) \cdot \axy}{\mathbf{1} \cdot \axy}  \right) \mathbf{1}.\]
		Then $\bq \cdot \axy = (\bp - T(\bp)) \cdot \axy$; i.e., $\left( T(\bp) + \bq \right) \cdot \axy = \bp \cdot \axy$.

		Since $T$ preserves the inner product and $T(\rxy) = \rxy$, we have $\bp \cdot \rxy = T(\bp)\cdot T(\rxy) = T(\bp)\cdot \rxy$; hence, $\left( T(\bp) - \bp\right) \cdot \axy = \left( T(\bp) - \bp\right) \cdot \ayx$.  Therefore
		\begin{eqnarray*}
		\bq \cdot \ayx &=& \bq \cdot \axy  \; \; \; \; \; \; \; \; \; \;  (\text{ Proposition \ref{rproperties} (3)})\\
		&=& (\bp - T(\bp)) \cdot \axy \\
		&=& (\bp - T(\bp)) \cdot \ayx,
		\end{eqnarray*} 
		and so $\left( T(\bp) + \bq \right) \cdot \ayx = \bp \cdot \ayx$.
		\end{proof}
\noindent We are now prepared to restate the IIA criterion as a purely geometric one.  Recall that a rotation about a vector $\mathbf{r} \in \mathbb{R}^k$ is an element of $SO(k)^{\mathbf{r}}$.   
\begin{prop}\label{IIAequiv}
	The positional voting method $\bw$ satisfies IIA if and only if $SO(k)^{\rxy}$  preserves $(\vxy)_+$.
\end{prop}
\begin{proof}
	For convenience we abbreviate $\vxy = \bv$, and similarly for $\rxy$.  We will use the fact that $\bv \perp \mathbf{1}$ (Proposition \ref{vproperties}).  First assume $\bw$ satisfies IIA.  Let $T \in SO(k)^\br$, and assume $\bp \cdot \bv > 0$.  Then by Lemma \ref{rotationlemma}, $T(\bp) + \bq \sim_{X,Y} \bp$ for some $\bq \in \mathbb{R}\mathbf{1}$; hence, $T(\bp)\cdot \bv = (T(\bp) + \bq) \cdot \bv > 0$.

	Conversely, suppose rotation about $\br$ preserves $(\bv)_+$.  Then for any $T \in SO(k)^\br$ we have $T(\bv)_+ \subseteq (\bv)_+$.  By Lemma \ref{rotationproperties} (3), we have $T(\bv) = \bv$.  Since the only fixed points of $SO(k)^\br$ are elements of $\mathbb{R} \br$ (Lemma \ref{rotationproperties} (4)), we must have $\bv = c \, \br$ for some scalar $c$.  Now suppose $\bp \cdot \bv > 0$.  Then for any $\bq \in \left[ \bp \right]_{XY}$,  we have $(\bp - \bq) \cdot \br = 0$; hence, $\bq \cdot \bv = c\, \bq \cdot \br = c\,  \bp \cdot \br  = \bp \cdot \bv   > 0$.  
\end{proof}

\begin{thm}
	A nontrivial linear neutral CWF $F$ on a profile space $\prof$ satisfies IIA if and only if $|\lambda| = 2$.  
\end{thm}
\begin{proof}
	Let $F$ be a nontrivial linear neutral CWF.  By Theorem \ref{linearscf}, $F$ is a positional voting method. Rotation about $\rxy$ preserves $(\vxy)_+$ if and only if $\vxy = c \; \rxy$ for some scalar $c$.  If $| \lambda | > 2$, Proposition \ref{vproperties} guarantees $\vxy = 0$; hence, $F$ is trivial.  If $| \lambda |  = 2$, Proposition \ref{vproperties} gives us $\vxy \parallel \rxy$, and by Proposition \ref{IIAequiv} $F$ satisfies IIA.
\end{proof}
	Let a group $G$ act upon $\prof$.  We will say that a profile $\bp \in \prof$ is \textit{IIA--unstable with respect to $G$} if there exists some $g \in G$ such that $\bp$ and $g.\bp$ provide an IIA--violation.  The result above demonstrates that there exist profiles $\bp$ which are IIA--unstable with respect to the action of the special orthogonal group on $\prof$.  Of course, if $\bp$ is IIA--unstable with respect to $G$ and $G$ is a subgroup of $H$, then $\bp$ is also $IIA$--unstable with respect to $H$.  The converse, however, should be false, and this motivates an interesting question - if $\bp$ is IIA--unstable with respect to $G$ and $G' \leq G$, does (or when does) $\bp$ remain IIA--unstable with respect to $G'$?
\subsection{Pareto efficiency}

			Classically, a candidate $X$ is preferred \textit{unanimously} to a candidate $Y$ in a (nonnegative) profile $\bp$ if $X$ appears `above' $Y$ for all ballots in $\bp$.  The criterion of Pareto efficiency demands that $X$ defeats $Y$ in an election whenever $X$ is preferred unanimously to $Y$.  To adapt the condition of unanimous preference to an arbitrary profile in a manner consistent with Postulate 1, we will say that a (nonzero) profile $\bp$ \textit{prefers $X$ unanimously to $Y$} if 
			
			\begin{itemize}
				\item $\bp(b) \geq 0$ for any ballot $b$ satisfying $b(X) < b(Y)$, 
				\item $\bp(b) \leq 0$ for any ballot $b$ satisfying $b(X) > b(Y)$, 
				\item the projection of $\bp$ onto $P_{\mathbf{C}, \lambda}^{X>Y}  \oplus P_{\mathbf{C}, \lambda}^{X<Y}$ is nonzero. 
		 	\end{itemize}
		
			 The collection of all profiles which prefer $X$ unanimously to $Y$, which we will denote $\left( \prof \right)_{X \gg Y}$, is then 			
			 \begin{eqnarray*}
			  \left( \prof \right)_{X \gg Y} = \left\{ \sum_{b(X) < b(Y) } r_b \delta_b \; \; +  \sum_{b(X) > b(Y) }  s_b \delta_b \; \; +   \sum_{b(X) = b(Y) }  t_b \delta_b   : \right. \\
			   \left.    r_b \geq 0 , \; s_b \leq 0, \; \sum r_b - \sum s_b > 0 \right\} . 
			  \end{eqnarray*}
			If $\lambda = (1, \ldots, 1)$, then $\left( \prof \right)_{X \gg Y}$ is an orthant of $\prof$. 
				
			Our adapted definition for Pareto efficiency is this: a CWF $F$ is \textit{Pareto efficient} if, for any profile $\bp$ which prefers $X$ unanimously to $Y$, we have $F(\bp)(X) > F(\bp)(Y)$.  This equality holds if and only if $\bp \cdot \vxy > 0$. So we see that $F$ is Pareto efficient if and only if $\left( \prof \right)_{X \gg Y}$ lies in the half--space $\left( \vxy \right)_+$.  
			
		\begin{prop}\label{paretocriterion}
		The positional voting method $\bw$ is Pareto efficient if and only if $\bfw$ is strictly decreasing; i.e., $\bfw(i) > \bfw(i+1)$.	
		\end{prop}
		\begin{proof}
			First assume $\bfw$ is strictly decreasing.  We must show that $\left( \prof \right)_{X \gg Y} \subseteq \left( \vxy \right)_+$, for which it is necessary and sufficient to show the containment 
			\[ \displaystyle{ \left\{ \delta_b \right\}_{ b(X) < b(Y)} \cup \left\{ - \delta_b \right\}_{b(X) > b(Y)}} \subseteq \left( \vxy \right)_+.\]  If $b(X) < b(Y)$, then  
			\begin{equation}\label{paretoineq1} 0 < \bfw(b(X)) - \bfw(b(Y)) = \delta_b \cdot \vxy.\end{equation}
			If $b(X) > b(Y)$, then 
			\begin{equation}\label{paretoineq2} 0 < \bfw(b(Y))  - \bfw(b(X)) = (-\delta_b) \cdot \vxy.\end{equation} 
			Conversely, if  $\left( \prof \right)_{X \gg Y} \subseteq \left( \vxy \right)_+$, we obtain the inequalities $\delta_b \cdot \vxy > 0$ for $b(X) < b(Y)$, from which $\bfw(i) > \bfw(i+1)$ follows. 
		\end{proof}
		\subsection{The strong majority criterion}
		The classical majority criterion is defined as follows: suppose $\mathbf{p}$ is a nonnegative profile in $\prof$ for $\lambda = (1, \ldots, 1)$.   A \textit{majority candidate} is one who receives a majority of first--place votes.  This candidate may not exist, but if they do they are unique.  A social choice function satisfies \textit{the majority criterion} if the majority candidate, when she exists, is the unique winner of the election.  There are some technical obstructions to adapting this definition to the present framework, notably due to the existence of negative coefficients in an arbitrary profile.  Here we discuss a variation of this criterion, which can be applied to arbitrary partitions and profiles.  
\begin{defn}
	A CWF $F$ satisfies the \textit{strong majority criterion} if, whenever a candidate $X$ defeats a candidate $Y$ in a head--to--head race in a profile $\bp$, we have $F(\bp)(X) > F(\bp)(Y)$; i.e., $X$ defeats $Y$ in the election.  
	\end{defn}
For nonnegative profiles, the strong majority criterion clearly implies the classical majority criterion.

By Theorem $\ref{linearscf}$, any CWF $F$ is a positional voting method $\bw$.  Therefore we can recast the above definition as follows. Note that candidate $X$ defeats candidate $Y$ in a head--to--head race if and only if $\bp \cdot \rxy > 0$, and  $X$ defeats $Y$ in the election if and only if $\bp \cdot \vxy > 0$.  So we have the following geometric characterization: 

\begin{prop}
The positional voting method $\bw$ satisfies the strong majority criterion if and only if $( \rxy )_+ \subseteq (\vxy)_+$. 	
\end{prop}
For $|\lambda| > 2$, this criterion is - like IIA - a terrible one, because it is too strong - it is \textit{never} satisfied.  Of interest is not the result in itself, perhaps, but the geometric characterization and proof.

 \begin{thm}   If $| \lambda | = 2$, then the positional voting method $\bw$ satisfies the strong majority criterion if and only if $\bfw(1) > \bfw(2)$.  If  $| \lambda | > 2$ then no nontrivial positional voting method satisfies strong majority.
\end{thm}
\begin{proof}
  First assume $|\lambda| = 2$.  In this case, the vectors $\rxy$ and $\vxy$ are parallel - in fact, $\vxy = (\bfw(1) - \bfw(2)) \rxy$ (see proof of Proposition \ref{vproperties}).  If $\bfw(1) > \bfw(2)$, then $\vxy = c \rxy$ for $c > 0$; hence, $(\rxy)_+ = (\vxy)_+$, and Proposition \ref{paretocriterion} guarantees that $\bw$ is Pareto efficient. Conversely, if $\bw$ is Pareto efficient, Proposition \ref{paretocriterion} guarantees that $\bfw(1) - \bfw(2) >~0$.  
	
	Next assume $|\lambda| > 2$, and assume  $F$ is a nontrivial CWF.  By Theorem \ref{linearscf}, $F$ is a positional voting method.  By Proposition \ref{halfspaces}, $( \rxy )_+ \subseteq (\vxy)_+$ only if $\rxy \parallel \vxy$.  By Proposition \ref{vproperties}, this can only happen if $\vxy = 0$, which yields the trivial CWF.
\end{proof}
\noindent By Theorem \ref{linearscf}, all CWF's with $| \lambda | > 2$ violate the strong majority criterion.  
	 
 \subsection{The Condorcet criterion}
	A candidate $X \in \mathbf{C}$ is a Condorcet candidate in a profile if, given any other candidate $Y \in \mathbf{C}$, $X$ defeats $Y$ in a head--to--head race.  A social choice function $F$ satisfies the Condorcet criterion if the Condorcet candidate is guaranteed (unique) victory.  Although we don't have a characterization of all neutral linear CWFs that satisfy the Condorcet criterion, we do provide an interesting geometric characterization of this condition.  Before stating the condition, we need a lemma.  
	 \begin{lemma}\label{uxy}
	 	Let $F$ be a neutral linear CWF, and let $\uxy = F(\rxy)$.  Then $F(\bp)(X) > F(\bp)(Y)$ if and only if $F(\bp) \cdot \uxy > 0$.
	 \end{lemma}
	 
	 \begin{proof}
	 	By Theorem \ref{linearscf}, $F = \bw$ for some $\mathbf{w} \in \mathbb{R}^{\mathbf{C}}$.  To prove the result, it is sufficient to show there exists some $c \in \mathbb{R}_+$ such that 
	 	\[  \uxy(Z)  = \begin{cases}
	 	c, & Z = X \\
	 	-c, &  Z   = Y \\
	 	0, & \text{ otherwise. }
	 	\end{cases} \] 
	 	For $Z \in \mathbf{C}$, we have 
	 	\begin{eqnarray*}
	 		\uxy(Z)  &=& F(\rxy)(Z) =  \bw(\rxy)(Z)\\
	 		&=& \rxy \cdot \bv_{\scriptscriptstyle{Z}}\\
	 		&=& \sum_{b \in \mathbf{C}_\lambda}\rxy (b)( \bfw \circ \ev_Z) (b)\\
	 		&=& \sum_{b \in \mathbf{C}_\lambda^{X>Y}} \mathbf{w}(b(Z)) - \sum_{b \in \mathbf{C}_\lambda^{X<Y}} \mathbf{w}(b(Z)).
	 	\end{eqnarray*}
	 	If $Z \neq X, Y$, then $\uxy(Z) = 0$, and clearly we have $\uxy(X) = -  \uxy(Y)$.  So it only remains to show that $\uxy(X) > 0$. 
	 	For $i,j \in [m]$, define $\clxy{i}{j} = \{b \in \mathbf{C}_{\lambda} \, | \, b(X) = i, b(Y) = j\}$. In other words, $\clxy{i}{j}$ is the set of all ballots that have $X$ ranked in level $i$ and $Y$ ranked in level $j$. Note that $\left|\clxy{i}{j}\right| = \left|\clxy{j}{i}\right|$ for all $i,j$. Now, we may write 
	 	\begin{eqnarray*} 
	 		\uxy(X) & = & \dsum_{i = 1}^{m-1} \dsum_{j = i+1}^{m} \left(\left|\clxy{i}{j}\right|\bfw(i) - \left|\clxy{j}{i}\right|\bfw(j)\right) \\ 
	 		& = & \dsum_{i = 1}^{m-1} \dsum_{j = i+1}^{m} \left|\clxy{i}{j}\right|\left(\bfw(i)-\bfw(j)\right).
	 	\end{eqnarray*}
	 	
	 \end{proof}

	A candidate $X$ is a Condorcet candidate in a profile $\bp$ if and only if, for any $Y \neq X$, $ \bp \cdot \rxy > 0$ (Postulate 1).  If $Y \neq Z$, we have $\br_{\scriptscriptstyle{X>Z}} = (Y \; Z).\rxy$.  Therefore we can write 
		 \[  \bp \cdot \br_{\scriptscriptstyle{X>Z}} = \bp \cdot (Y \; Z).\rxy = (Y \; Z).\bp \cdot \rxy,\] 
		 and we see that the original condition $ \bp \cdot \rxy > 0 $ for all $ Y \neq X$ is equivalent to $S^X_{\mathbf{C}}.\bp \subseteq (\rxy)_+$. 
		  The candidate $X$ is the unique winner if and only if $F(\bp) \cdot \uxy > 0$ for all $Y \neq X$, where $\uxy$ was introduced in Lemma \ref{uxy} above; this condition is equivalent to $S^X_{\mathbf{C}}.F(\bp) \subseteq (\uxy)_+$.  Since $F$ is $S_{\mathbf{C}}$--equivariant, this is equivalent to $F(S^X_{\mathbf{C}}. \bp) \subseteq (\uxy)_+$.  If $F$ is realized as a positional voting method $\bw$, we can alternatively state this `unique winning' condition as $S^X_{\mathbf{C}}.\bp \subseteq (\vxy)_+$.  We summarize these observations in the following proposition. 
		 
	\begin{prop}
		A candidate $X$ is a Condorcet candidate in the profile $\mathbf{p}$ if and only if, given any other candidate $Y \in \mathbf{C}$, the $S_{\mathbf{C}}^{X}$--orbit of $\mathbf{p}$ is contained in $( \rxy )_+$.  The CWF $F$ satisfies the Condorcet criterion if and only if 
				\begin{equation}\label{condorcet1} S^X_{\mathbf{C}}.\bp \subseteq (\rxy)_+ \; \Rightarrow \; F(S^X_{\mathbf{C}}.\bp) \subseteq (F(\rxy))_+. \end{equation}
			If $F$ is realized as a positional voting method $\bw$, then this condition is equivalent to 
			\[  S^X_{\mathbf{C}}.\bp \subseteq (\rxy)_+ \; \Rightarrow \;  S^X_{\mathbf{C}}.\bp \subseteq (\vxy)_+. \] 
	\end{prop}

\section{Equivalent positional voting methods}
 In this section we describe classes of all CWFs that are `essentially the same'.  Although some of these results were previously given in \cite{orrison}, we offer an alternate proof of this classification. We present two equivalences on the collection of CWFs: order--equivalence and cardinal--equivalence. 

We say that two CWFs $F, G$ are \textit{order--equivalent}, written $F~\sim_o~G$, if, for any profile $\mathbf{p}$ and any candidates $X, Y$, $F(\mathbf{p})(X) > F(\mathbf{p})(Y)$ if and only if $G(\mathbf{p})(X) > G(\mathbf{p})(Y)$.  This means that for any profile, $F$ and $G$ will always return the same ordinal ranking of candidates (even if the numerical scores may differ). We extend this equivalence relation to $ \mathbb{R}^{\left[ m \right]}$: we say two vectors $\mathbf{u}$, $\bfw$ are order-equivalent, written $\mathbf{u} \sim_o \bfw$, if $B_{\mathbf{u}} \sim_o B_{\mathbf{w}}$.  It is easily verified that $\sim_o$ is an equivalence relation.   

We next define cardinal--equivalence.  The rationale behind this equivalence is the following: suppose $\mathbf{C} = \left\{ X, Y, Z\right\}$, and CWFs $F$ and $G$ satisfy, for some profile $\bp$, 
\[  F(\bp) =  (4, 1, 2), \; \; \; \; \; \; G(\bp) = (8, 2, 4).\]
These two outcomes are `essentially the same' for the three candidates, in the sense that the \textit{ratios} of points awarded are the same in each case ($X$ earns four times as many points as $Y$, and twice as many points as $Z$).  In this sense we should not distinguish between $F$ and $\alpha F$, if $\alpha$ is any positive scalar.  Consider now the outcomes 
\[ F(\bp) = (4, 1, 2), \; \; \; \; G'(\bp) = (10, 4, 6) = G(\bp) + (2, 2, 2). \]
In this case $F$ and $G'$ are proportional \textit{up to addition of a scalar multiple of $\mathbf{1}$} (in this case $(2,2,2)$).  As addition of any such vector should not affect the outcome of an election, we should not distinguish $G$ from $G'$.  So we say that $F$ and $G$ are \textit{cardinal--equivalent}, written $F \sim_c G$, if there exist $\alpha \in \mathbb{R}_{> 0}, r \in \mathbb{R}$ such that, for any profile $\bp$, $F(\bp) = \alpha( G(\bp) + r\mathbf{1})$.  
 We similarly extend the notion of cardinal--equivalence to $\mathbb{R}^{\left[ m \right]}$, and we denote by $\left[ \bfw \right]_c$ the equivalence class of $\bfw$ with respect to $\sim_c$.    Clearly cardinal--equivalence implies order--equivalence.

	\begin{lemma}\label{kernelT}
		Let $X$ and $Y$ be distinct candidates.  Let $T_{X>Y}: \mathbb{R}^{\left[ m \right]} \rightarrow   \prof$ be the linear map given by $T_{X>Y}: \bfw \mapsto  \bv_{\scriptscriptstyle{\mathbf{w}, X>Y}}$.  Then the kernel of $T_{X>Y}$ is the 1--dimensional subspace spanned by  $\mathbf{1} \in $ $ \mathbb{R}^{\left[ m \right]}$. 
	\end{lemma}
	
	\begin{proof}
	   An easy calculation shows $\mathbf{1} \in \ker(T_{X>Y})$.  For the opposite containment, recall that we may write $\vect{X}$ as \[\vect{X} = \sum_{b\in\mathbf{C}_{\lambda}} \bfw(\mathbf{b}(X))\delta_{\mathbf{b}}.\]  Then we have \[T_{X>Y}(\bfw) = \sum_{b\in\mathbf{C}_{\lambda}} (\bfw(\mathbf{b}(X))-\bfw(\mathbf{b}(Y)))\delta_{\mathbf{b}}.\] For any $\mathbf{b}\in\mathbf{C}_{\lambda}$, this gives $T_{X>Y}(\bfw)(\mathbf{b}) = \bfw(\mathbf{b}(X)) - \bfw(\mathbf{b}(Y))$. 
	   Suppose that $\bfw \in\ker(T_{X>Y})$, let $i, j \in \left[  m \right]$, and let $\mathbf{b} \in \clxy{i}{j}$. Then $ 0 = T_{X>Y}(\bfw)(\mathbf{b}) = \bfw(i) - \bfw(j)$; hence, $\bfw(i) = \bfw(j)$.
	
	\end{proof}

The first part of the next proposition shows that, up to order--equivalence, we can always assume that the weight vector $\bfw$ is orthogonal to $\mathbf{1}$ in $\mathbb{R}^{\left[ m \right]}$. The second part describes the order--equivalence class of $\bfw \in \mathbb{R}^{\left[ m \right]}$. 
\begin{prop}\label{orderequiv} 
Let $\bfw \in \mathbb{R}^{\left[ m \right]}$. 
\begin{enumerate} 
\item The order--equivalence class of $\bfw$ always contains an element $\bfw' \in \mathbf{1}^\perp$. 
\item Let $\bfw \in \mathbf{1}^\perp$.  The order--equivalence class of $\bfw$ is the positive half--plane $(\bfw)_+$ of $\displaystyle{\mathbb{R}\mathbf{1} \oplus \mathbb{R}\bfw }$.
\end{enumerate}
\end{prop}
\begin{proof}  $B^\lambda_{\mathbf{u}} \sim_o B^\lambda_{\mathbf{w}}$ if and only if, for any profile $\mathbf{p}$ and any candidates $X, Y$, 
	 \[ \mathbf{p}\cdot \mathbf{v}_{\scriptscriptstyle{\mathbf{u}, X>Y}} > 0  \; \; \; \; \; \Leftrightarrow \; \; \; \; \; \mathbf{p}\cdot \mathbf{v}_{\scriptscriptstyle{\mathbf{w}, X>Y}} > 0 . \]   This is true if and only if $ \mathbf{v}_{\scriptscriptstyle{\mathbf{u}, X>Y}}= \alpha  \mathbf{v}_{\scriptscriptstyle{\mathbf{w}, X>Y}}$ for some $\alpha > 0$.  This, in turn, is true if and only if $\mathbf{u} - \alpha \bfw$ is in the kernel of $T_{X>Y}$; i.e., $\mathbf{u} - \alpha \bfw \in \mathbb{R}\mathbf{1}$ (by Lemma \ref{kernelT}).  Therefore the positional voting methods $B^\lambda_{\mathbf{u}}$	and $B^\lambda_{\mathbf{w}}$ are order--equivalent if and only if $\bu \in \mathbb{R}_{>0}\bfw \oplus \mathbb{R} \mathbf{1}$.  \end{proof}
	 
	 The weight vectors $\bfw$ and $-\bfw$ are certainly not order--equivalent: clearly $-\bfw$ fully \textit{reverses} the order corresponding to $\bfw$.  For this reason, we will refer to $\bw$ and $-\bw = B^\lambda_{-\mathbf{w}}$ as an \textit{antipodal pair}.  The Proposition states that the plane $\displaystyle{\mathbb{R}\mathbf{1} \oplus \mathbb{R}\bfw }$ contains the order--equivalence classes for $\bfw$ and $-\bfw$  (and all such planes contain the order--equivalence class for the trivial voting method).  
	\begin{center}
\includegraphics[scale=0.7]{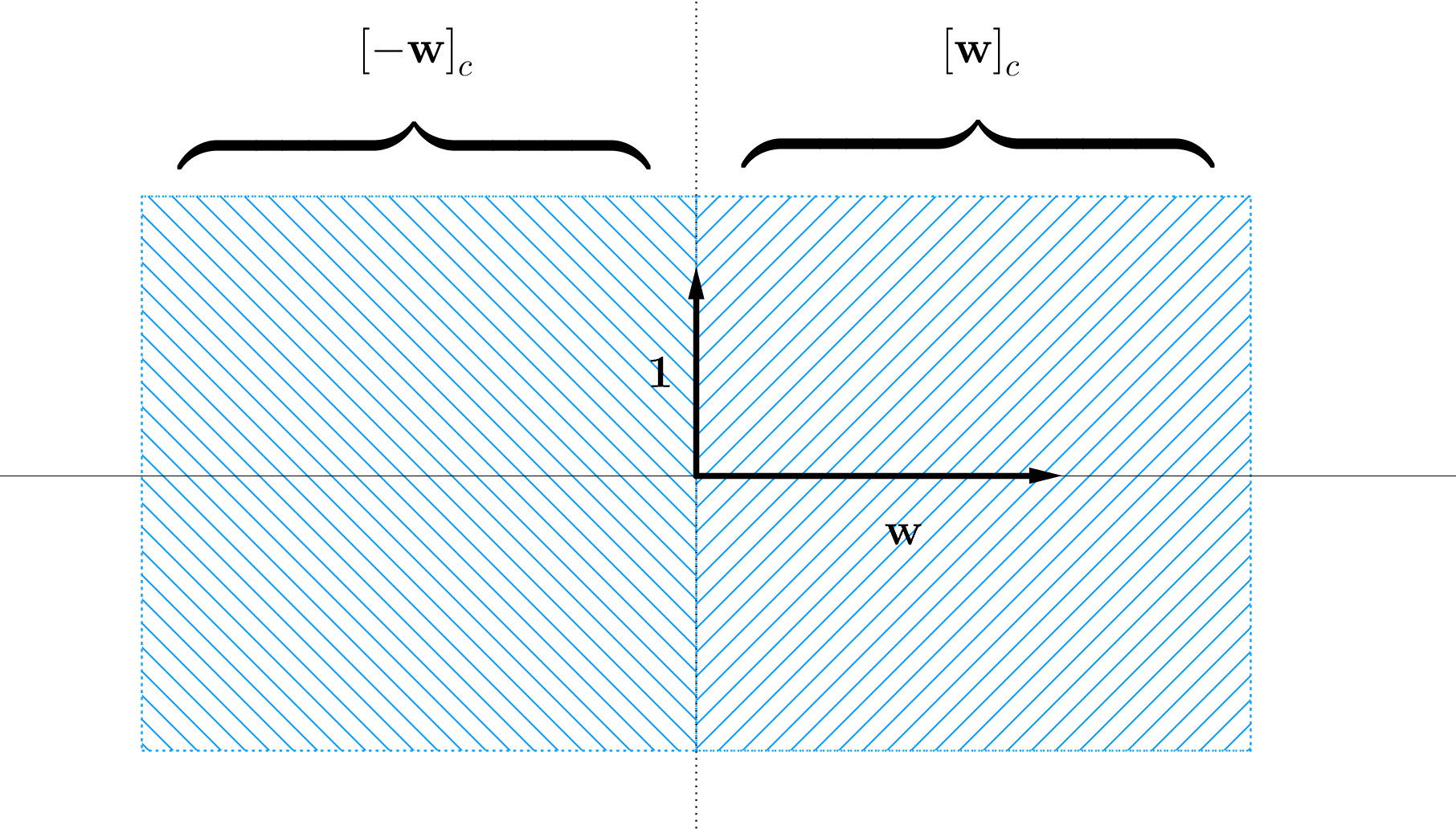}
\end{center}
	 \begin{cor} \mbox{}
	 \begin{enumerate}
	 	\item 	 Up to order--equivalence, the collection of all antipodal pairs $\pm F$ of linear CWFs on $\prof$ is parametrized by the projective space $\mathbb{R}\mathbb{P}^{m- 2}$.
	   \item Two linear CWFs are order--equivalent if and only if they are cardinal--equivalent. 
	    \end{enumerate} 
	 	 \end{cor}
	 \begin{proof}
	 	The collection of order--equivalence classes in $\mathbb{R}^{\left[ m \right]}$ for $\pm \bw$ consists of the family of all planes which contain (hence intersect at) the subspace $\mathbb{R}\mathbf{1}$.  We note that this is also the Grassmannian $ G(1, m - 1)$.
	 	
	 	Let $F = B^\lambda_{\mathbf{u}}$ and $G = \bw$.  If $B^\lambda_{\mathbf{u}} \sim_c \bw$, then clearly $B_{\mathbf{u}} \sim_o \bw$.  Conversely, suppose $B_{\mathbf{u}} \sim_o \bw$.  Then by the Proposition \ref{orderequiv} we have $\bfw = \alpha \bu + \beta \mathbf{1}$.  By linearity of $\bw$ (in $\bfw$) we have $\bw = B_{ \alpha \bu + \beta \mathbf{1}} = \alpha B^\lambda_{\mathbf{u}} + \beta B^\lambda_{\mathbf{1}}$; hence, $ \bw \sim_c  B^\lambda_{\mathbf{u}}$. 
	 	\end{proof}
The diagram below illustrates the unit vector $\mathbf{1}$, along with the half-planes of equivalence classes corresponding to a collection of nonequivalent weight vectors (these weight vectors are not shown in the image).  Each half-plane, arranged as a page along the spine of $\mathbb{R}\mathbf{1}$, corresponds to exactly one equivalence class of CWF's.  
\begin{center}
\includegraphics[scale=0.4]{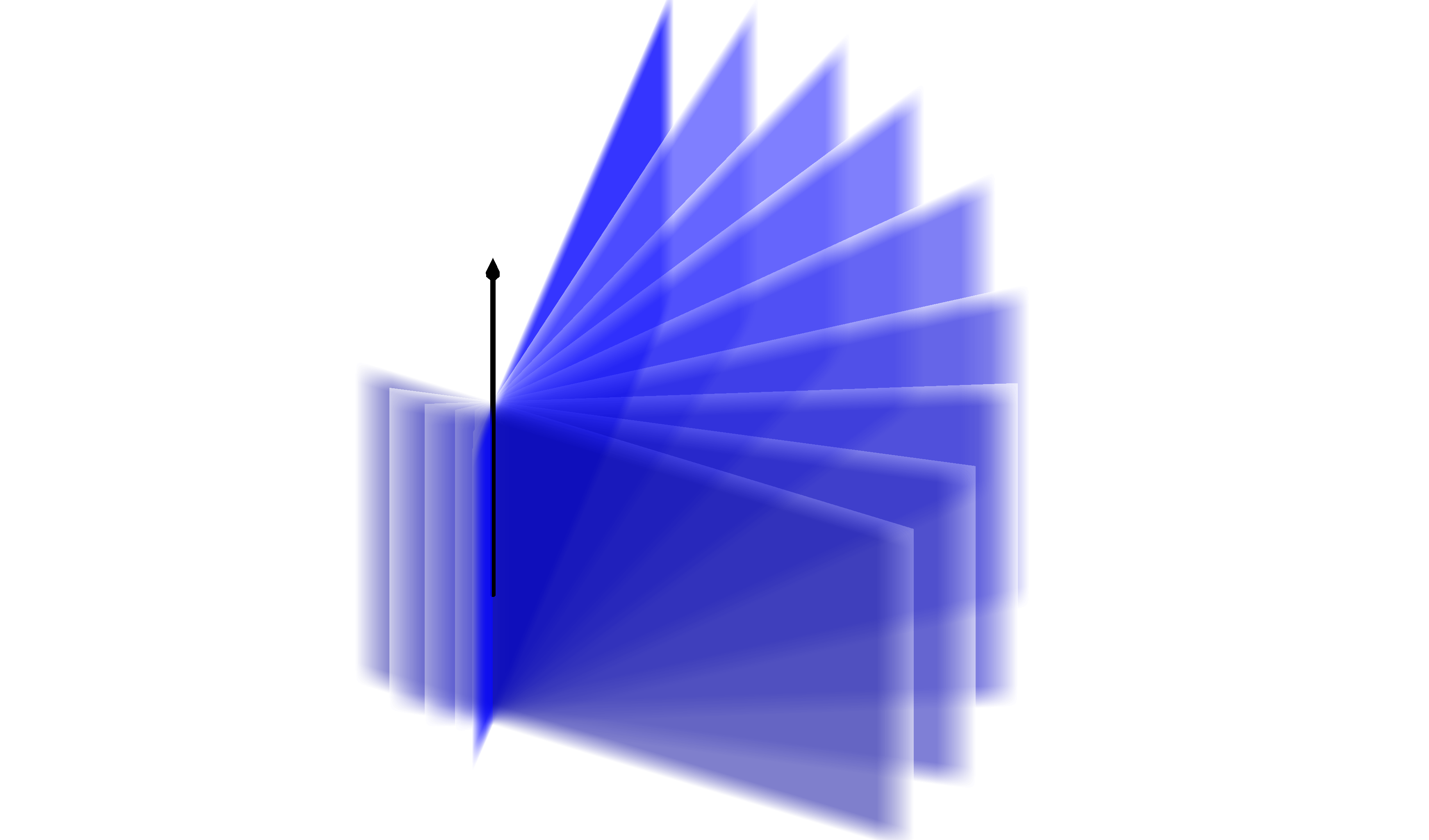}
\end{center}
\mbox{}
\\
\\
\\
 As mentioned above, the result in Proposition \ref{orderequiv} (2) was given in \cite{orrison}.  However, the equivalence of order-- and cardinal--equivalence was not mentioned there.  The identification of antipodal pairs of CWFs with a projective space is an easy corollary, but it seems worth mentioning; on a speculative note it may be interesting to use this identification to provide the collection of all CWFs with this natural topology for future applications.

\end{document}